\documentclass[a4paper,reqno,11pt]{amsart}
\usepackage{t1enc}
\usepackage[margin=1in]{geometry}
\usepackage{ifthen}
\usepackage{graphicx}

\newcommand{\R}{\mathbf{R}}

\newcommand{\ex}{\mathbf{E}}

\newcommand{\rk}{\textrm{rk}}

\newcommand{\ind}{\mathbf{1}}

\usepackage{color}
\newcommand{\pg}[1]{#1}

\theoremstyle{plain}
\newtheorem{theorem}{Theorem}

\newtheorem{proposition}{Proposition}

\theoremstyle{definition}
\newtheorem{definition}{Definition}

\newtheorem{remark}{Remark}

\theoremstyle{remark}

\newcommand{\formula}[2][nolabel]
{\ifthenelse{\equal{#1}{nolabel}}
 {\begin{align*} #2 \end{align*}}
 {\ifthenelse{\equal{#1}{}}
  {\begin{align} #2 \end{align}}
  {\begin{align} \label{#1} #2 \end{align}}
 }
}

\numberwithin{equation}{section}

\begin{document}

%
%

\title [On squared Bessel particle systems]{On squared Bessel particle systems}
\thanks{Jacek Ma{\l}ecki was supported by NCN grant no. 2013/11/D/ST1/02622}
\author{Piotr Graczyk, Jacek Ma{\l}ecki}
\address{Piotr Graczyk \\ LAREMA \\ Universit\'e d'Angers \\ 2 Bd Lavoisier \\ 49045 Angers cedex 1, France}
\email{piotr.graczyk@univ-angers.fr}
\address{  Jacek Ma{\l}ecki,  \\  Faculty of Pure and Applied Mathematics\\ Wroc{\l}aw University of Science and Technology \\ ul. Wybrze{\.z}e Wyspia{\'n}\-skiego 27 \\ 50-370 Wroc{\l}aw, Poland}
\email{jacek.malecki@pwr.edu.pl }

\begin{abstract}
We study the existence and uniqueness of SDEs describing squared Bessel particles systems in full generality. We define non-negative and non-colliding squared Bessel particle systems and we study their properties. 
\end{abstract}
\maketitle
%
%
\section{Introduction}

The main \pg{objective}  of the paper is to study in details the following system of stochastic differential equations
\formula[eq:BESQp:SDE]{
  d X_i &= 2\sqrt{|X_i|}dB_i + \left(\alpha+\sum_{j\neq i}\frac{|X_i|+|X_j|}{X_i-X_j}\ind_{\{X_i\neq X_j\}}\right)dt\/,\quad i=1,\ldots,p\\
	&X_1(t)\leq X_2(t)\leq\ldots\leq X_p(t),\quad\quad t\geq 0\/, 
}	
with the initial condition $X_i(0)=x_i$, $i=1,\ldots, p$ and $\alpha\in\R$. The system \eqref{eq:BESQp:SDE} is called \textit{squared Bessel particle system} following the fact that for $p=1$ it reduces to the classical squared Bessel stochastic differential equation
\formula[eq:BESQ:SDE]{
   dX = 2\sqrt{|X|}dB+\alpha dt\/,\quad X(0)=x\/.
}
It follows from the Yamada-Watanabe theorem \cite{bib:yw71} that there exists a unique strong solution to \eqref{eq:BESQ:SDE} and the solution is called squared Bessel process of dimension $\alpha$ starting from $x$. It is usually denoted by $BESQ^{(\alpha)}(x)$. In the classical setting the non-negativity of $\alpha$ and $x$ are assumed. However, G\"oing-Jaeschke and Yor studied squared Bessel processes starting from negative points as well as having negative dimensions (see \cite{bib:gyj}), that play an important role in the stochastic calculus in one dimension. The present paper generalizes the G\"oing-Jaeschke-Yor's description of squared Bessel processes to the multidimensional case. 

On the other hand, by \cite[Theorem 3]{bib:gm1}, the system \eqref{eq:BESQp:SDE}  describes the ordered eigenvalues of the solution to the following matrix stochastic differential equation
\formula[eq:Wishart:SDE]{
   d\mathbb{Y}_t = \sqrt{|\mathbb{Y}_t|}d \mathbb{W}_t+d \mathbb{W}^T_t\sqrt{|\mathbb{Y}_t|}+\alpha \mathbb{I}dt\/,
	}
where $\mathbb{Y}_t\in {\mathcal S}_p$, the vector space of real symmetric matrices, 
$ \mathbb{W}_t$ is a Brownian $p\times p$ matrix and the eigenvalues of $ \mathbb{Y}_0$ are all different.
The equation \eqref{eq:Wishart:SDE} is usually considered with the additional assumption $\alpha \geq p-1$ (which for $p=1$ corresponds to the condition $\alpha\geq 0$),  and then it is called 
 \textit{Wishart SDE}, which can be viewed as the matrix generalization of the squared Bessel SDE \eqref{eq:BESQ:SDE} (see \cite{bib:b91,bib:donatiyor,bib:gm1}). If ${\alpha \geq p-1}$ and   the eigenvalues $X_1(0),\ldots,X_p(0)$  of $ \mathbb{Y}_0$ are supposed to be non-negative, then the particles $X_i(t)$ (\pg{i.e. the} eigenvalues of $\mathbb{Y}_t$) remain non-negative (i.e. $X_1(t)\geq 0$ for every $t\geq 0$) and they never collide for $t>0$.  In fact in this case we can remove absolute values and the indicators from \eqref{eq:BESQp:SDE} and \eqref{eq:Wishart:SDE}. However, the matrix equations  \eqref{eq:Wishart:SDE} are also considered without any restrictions on $\alpha$ and the behaviour of their eigenvalues for $\alpha<p-1$ is of great importance (see \cite{bib:gmm:2017}).

Since the squared Bessel particle systems  \eqref{eq:BESQp:SDE} were not studied for arbitrary $\alpha\in\R$ and $X(0)$, we provide results on the existence, unicity and properties of the solutions of the system \eqref{eq:BESQp:SDE} in the whole generality of its parameters and initial values. This general approach forces us to deal with some special values of $\alpha\in\R$ and $X(0)$ for which the unicity of solutions does not hold. It makes the study much more complicated than in the one dimensional case studied in \cite{bib:gyj}. Our results are partially based on the theory built in \cite{bib:gm2}, which allows to construct non-colliding solutions to general particle systems. However, there are some special cases of $\alpha$ and starting points $X(0)$ in  \eqref{eq:BESQp:SDE}, for which the results of  \cite{bib:gm2} cannot be applied directly. These cases require more in-depth analysis.

\pg{Systems of stochastic differential  equations with the indicators $\ind_{\{X_i\neq X_j\}}$ in the drift part
were introduced by Katori (\cite[Theorem 1]{bib:katori2004}, \cite{bib:katori2016}), but he uniquely considered  cases when
one can omit these indicators.}	
	

Note that the results obtained for the solutions of  the system \eqref{eq:BESQp:SDE} may be generalized for the $\beta$-BESQ particle systems, obtained by multiplying  the drift term in  \eqref{eq:BESQp:SDE} by a $\beta>1$, see \cite[Section III.D]{bib:gm1}. When $\beta=2$, these are the SDEs for $p$ independent BESQ processes on $\R^+$, conditioned not to collide (\cite{koc}). Such $\beta$-generalization of the present study will be done in the upcoming paper. 

The paper is organized as follows. We begin with introducing definitions and notations for non-colliding and non-negative solutions of \eqref{eq:BESQp:SDE} together with the results on their existence and uniqueness (Theorems \ref{thm:SDEs:NC:unicity} and \ref{thm:SDEs:Nonnegative}). In Theorem \ref{thm:SDEs:general} we give necessary and sufficient conditions on parameters of squared Bessel particle system to have a unique strong solution. In Section \ref{sec:poly} we study in details the stochastic description of the symmetric polynomials related to the non-negative solutions, which are used in the next section, where the proofs of the main results are provided. Finally, in Section \ref{sec:structure} we describe the structure of non-colliding solutions.


\section{Existence and uniqueness of solutions of BESQ particle system}

We start our considerations with studying so-called non-colliding solutions.

\begin{definition}
  A solution $(X_1,\ldots,X_p)$ of \eqref{eq:BESQp:SDE} is called \textbf{non-colliding} if there are no collisions between particles after the start, i.e. 
	\begin{eqnarray*}
	   T = \inf\{t>0: X_i(t)=X_j(t)\textrm{ for some $i\neq j$}\}
	\end{eqnarray*}
	is infinite almost surely. 
\end{definition} 

It appears that we can always build a non-colliding solution of \eqref{eq:BESQp:SDE} and  uniqueness among non-colliding solutions \pg{holds}, which is provided in the following

\begin{theorem}
\label{thm:SDEs:NC:unicity}
For every $\alpha\in\R$ and $x_1\leq \ldots\leq x_p$ there exists a unique \textbf{non-colliding} strong solution to the system of stochastic differential equations
\formula[eq:SDE:BESQnc]{
  d X_i &= 2\sqrt{|X_i|}dB_i + \left(\alpha+\sum_{j\neq i}\frac{|X_i|+|X_j|}{X_i-X_j}\ind_{\{X_i\neq X_j\}}\right)dt\/,\quad i=1,\ldots,p\\
	&X_1(t)\leq X_2(t)\leq\ldots\leq X_p(t)\quad t\geq 0
}
with the initial condition $X_i(0)=x_i$ for $i=1,\ldots,p$.
\end{theorem}

The proof of Theorem \ref{thm:SDEs:NC:unicity} is postponed until Section \ref{sec:EandU}, since it requires some knowledge of \pg{elementary} symmetric polynomials studied in details in Section \ref{sec:poly}.

\begin{remark}
  Note that if we study non-colliding solutions we can remove the indicators from the drift parts of equations   \pg{\eqref{eq:SDE:BESQnc}}.
\end{remark}

Theorem \ref{thm:SDEs:NC:unicity} enables us to introduce the following

\begin{definition}
   The unique strong solution to \eqref{eq:SDE:BESQnc}, which has no collisions after the start is called \textbf{non-colliding squared Bessel particle system} of dimension $\alpha\in\R$ starting from the point $(x_1,\ldots,x_p)$, where $x_1\leq x_2\leq \ldots\leq x_p$ and it will be denoted by $BESQ_{nc}^{(\alpha)}(x_1,\ldots,x_p)$.
\end{definition}

Since, by Theorem \ref{thm:SDEs:NC:unicity}, there always exists a unique non-colliding solution, it is natural to ask if there are any other solutions. To formulate the result providing necessary and sufficient conditions for \eqref{eq:BESQp:SDE} to have unique strong solution we have to introduce the following notation.

For fixed $p$ and $\alpha\in \{0,1,\ldots,p-2\}$ we define an integer number $n^*$ by requesting that
\begin{equation}\label{def:n_star}
2n^*\in \{p+\alpha,p+\alpha+1\}.
\end{equation} 
 Note that $n^*$ is  uniquely determined, since exactly one of the consecutive integer numbers is even. Moreover, for a fixed point $x=(x_1,\ldots,x_p)\in \R^n$, $x_1\leq \ldots\leq x_p$, we define 
\begin{eqnarray*}
   \textrm{rk}^+(x) = \sum_{i=1}^p \ind_{(0,\infty)}(x_i)\/,\quad \textrm{rk}^-(x) = \sum_{i=1}^p \ind_{(-\infty,0)}(x_i)\/,
\end{eqnarray*}  
and set $\textrm{rk}(x) = \textrm{rk}^+(x)+\textrm{rk}^-(x)$, i.e. $\textrm{rk}^+(x)$, $\textrm{rk}^-(x)$, $\textrm{rk}(x)$ is the  number of \pg{strictly} positive, strictly negative and all  non-zero values among $x_1,\ldots,x_p$.

\begin{theorem}
   \label{thm:SDEs:general}
	 There exists \textbf{unique strong solution} to 
	\begin{eqnarray*}
	  d X_i &=& 2\sqrt{|X_i|}dB_i + \left(\alpha+\sum_{j\neq i}\frac{|X_i|+|X_j|}{X_i-X_j}\ind_{\{X_i\neq X_j\}}\right)dt\/,\quad i=1,\ldots,p\\
	&&X_1(t)\leq X_2(t)\leq\ldots\leq X_p(t),\quad t\geq 0
	\end{eqnarray*}
	with the initial condition $X(0)=x$, where $x=(x_1,\ldots,x_p)$, if and only if one of the following conditions holds 
	\begin{itemize}
	   \item[(a)] $|\alpha|\notin \{0,1,\ldots,p-2\}$
		 \item[(b)] $|\alpha|\in \{0,1,\ldots,p-2\}$ and ($\textrm{rk}^+(x)>n^*$ or $\textrm{rk}^-(x)>p-n^*$).
	\end{itemize}
	Moreover, the unique solution is non-colliding. 
\end{theorem}
Obviously, the unique strong solution from Theorem \ref{thm:SDEs:general} must be, by Theorem \ref{thm:SDEs:NC:unicity}, non-colliding.  

Next we consider the problem of existence and uniqueness of non-negative solutions. The classical results related to $p=1$ say that the squared Bessel process $BESQ^{(\alpha)}(x)$ is non-negative if and only if $x\geq 0$ and $\alpha\geq 0$. In the multidimensional case we can ask analogous question introducing the following

\begin{definition}
 A solution $(X_1,\ldots,X_p)$ of \eqref{eq:BESQp:SDE} is called \textbf{non-negative} if $X_1(t)\geq 0$ for every $t>0$ a.s.
\end{definition}

Looking at the matrix interpretation of considered particle systems, non-negativity of $(X_1,\ldots,X_p)$ is equivalent to the condition saying that the corresponding matrix value process stays in $\overline{{\mathcal S}_p^+}$, where ${\mathcal S}_p^+$ is the open cone of positive definite symmetric matrices. The multidimensional result is provided in 

\begin{theorem}
\label{thm:SDEs:Nonnegative}
There exists unique strong \textbf{non-negative} solution to 
	\begin{eqnarray*}
	  d X_i &=& 2\sqrt{|X_i|}dB_i + \left(\alpha+\sum_{j\neq i}\frac{|X_i|+|X_j|}{X_i-X_j}\ind_{\{X_i\neq X_j\}}\right)dt\/,\quad i=1,\ldots,p\\
	&&X_1(t)\leq X_2(t)\leq\ldots\leq X_p(t)\quad t\geq 0
	\end{eqnarray*}
	with the initial condition $X(0)=x$, where $x=(x_1,\ldots,x_p)$ and $x_1\geq 0$, if and only if one of the following conditions holds 
	\begin{itemize}
	   \item[(a)] $\alpha\geq p-1$
		 \item[(b)] $\alpha\in \{0,1,\ldots,p-2\}$ and $\textrm{rk}(x)\leq \alpha$.
	\end{itemize}
\end{theorem}
\begin{remark}
Note that  Theorem \ref{thm:SDEs:Nonnegative} is a spectral analogue of the characterization of the Non-central Gindikin Set proved in \cite{bib:gmm:2017}.  
\end{remark}

\section{Symmetric polynomials of squared Bessel particles}
\label{sec:poly}
The elementary symmetric polynomials of $X=(X_1,\ldots,X_p)$ are defined by
\begin{eqnarray*}
   e_n(X)= \sum_{i_1<i_2<\ldots<i_n}X_{i_1}X_{i_2}\cdot\ldots\cdot X_{i_n}
\end{eqnarray*}
for every $n=1,2,\ldots,p$. We use the convention that $e_0(X)\equiv 1$ and $e_n(X)\equiv 0$ for $n>p$. Moreover, we write $e_{n}^{\overline{j_1},\overline{j_2},\ldots,\overline{j_m}}(X)$ for an incomplete elementary  symmetric polynomial 
\begin{eqnarray*}
   e_{n}^{\overline{j_1},\overline{j_2},\ldots,\overline{j_m}}(X) = \sum_{\stackrel{i_1<i_2<\ldots<i_n}{i_k\neq j_l}}X_{i_1}X_{i_2}\cdot\ldots\cdot X_{i_n}\/,
\end{eqnarray*} 
i.e. the sum of all products of length $n$
\pg{of different $X_i$'s}, not including any of $X_{j_1}, \ldots, X_{j_m}$. 


\begin{proposition}
If $X$ is a non-colliding solution of \eqref{eq:BESQp:SDE}, then $(e_1,\ldots,e_p)$ are semi-martingales described by
\begin{eqnarray}
\label{eq:poly:general}
d e_n(X) = \left(\sum_{i=1}^p |X_i|\pg{(e_{n-1}^{\overline{i}}(X))^2}\right)^{1/2}dV_n+\left(\sum_{i=1}^p\alpha e_{n-1}^{\overline{i}}(X)-\sum_{i< j}(|X_i|+|X_j|)e_{n-2}^{\overline{i}, \overline{j}}(X)\right)dt
\end{eqnarray}
for $n=1,\ldots,p$. Here $(V_1,\ldots,V_p)$ is a collection of one-dimensional Brownian motions such that
\begin{eqnarray}
\label{eq:poly:brackets}
d\left<e_n(X),e_m(X)\right> = \pg{4}\sum_{i=1}^p|X_i|e_{n-1}^{\overline{i}}(X)e_{m-1}^{\overline{i}}(X)dt\/.
\end{eqnarray}
\end{proposition}
\begin{proof}
	We apply \cite[Prop.3.1]{bib:gm2}.
	\end{proof}
\pg{The map $e=(e_1,\ldots,e_p)$ is} a diffeomorphism between $C_+=\{(x_1,\ldots,x_p)\in\R^p: x_1< x_2<\ldots<x_p\}$ and $e(C_+)$. Following \cite[Chapter 3]{bib:gm2}, we denote by $f:e(C_+)\longrightarrow C_+$ its inverse and note that $f$ can be continuously extended to
\begin{eqnarray*}
   f: \overline{e(C_+)}\stackrel{1-1}{\longrightarrow} \overline{C_+}\/.
\end{eqnarray*}
It implies that using the map $f$ we can write SDEs \eqref{eq:poly:general} and \eqref{eq:poly:brackets} only in terms of $e_1,\ldots,e_p$. The coefficients of those equations are continuous and the singularities of the form $(X_i-X_j)^{-1}$ disappear. In particular, there always exists a solution of those equations (see Proposition 3.2 in \cite{bib:gm2}). 

In the next theorem we write the coefficients \pg{of equations
\eqref{eq:poly:general} and
\eqref{eq:poly:brackets} 
 in a transparent way} in terms of $e_1,\ldots,e_p$ themselves (i.e. without incomplete polynomials and $X$). In order to shorten
the formulas,  we write $e_n$ instead of $e_n(X)$ and we set \pg{$e_r\equiv0$ if $r<0$ or $r>p$}.
\begin{theorem}
\label{thm:poly}
The elementary  symmetric polynomials of the non-colliding solution of \eqref{eq:BESQp:SDE} starting from $0\leq x_1\leq \ldots\leq x_p$ are semi-martingales described up to the first exit time $T=\inf\{t>0: X_1(t)<0\}$ by \pg{the following system of $p$ SDEs}
\begin{eqnarray}
\label{eq:poly:SDE}
   de_n = 2\left(\sum_{k=1}^{\pg{p}}(2k-1)e_{n-k}e_{n+k-1}\right)^{1/2}dV_n+(p-n+1)(\alpha-n+1)e_{n-1}dt\/,
\end{eqnarray}
where $V_n$ are one-dimensional Brownian motions such that
\begin{eqnarray}\label{eq:poly:bracks}
  d\left<e_n, e_{m}\right> = 4\sum_{k=1}^{\pg{p}} (m-n+2k-1)e_{n-k}e_{m+k-1}
\end{eqnarray}
for every $1\leq n\leq m\leq p$. 
\end{theorem}
\pg{
\begin{remark}
	The sum  
	in formula \eqref{eq:poly:SDE} has non-zero terms for
	$k=1,\ldots, K=\min(n, p+1-n)$ and the sum in 
	 \eqref{eq:poly:bracks} 
	for $k=1,\ldots, K=\min(n, p+1-m)$.
\end{remark}
}
\begin{proof}
Since we consider only $t<T$, we remove all the absolute values from \eqref{eq:poly:general} and \eqref{eq:poly:brackets}. We first compute  the drift part in equation \eqref{eq:poly:general}. It is easy to see that
\begin{eqnarray*}
	\sum_{i=1}^p e_{n-1}^{\overline{i}}(X) = (p-n+1)e_{n-1}(X)\/,
\end{eqnarray*}
since every product of length $n-1$ appears $p-(n-1)$ times in the last sum. Similarly, we have
\begin{eqnarray*}
	\sum_{i<j}(X_i+X_j)e_{n-2}^{\overline{i},\overline{j}}(X) = \sum_{i\neq j} X_i e_{n-2}^{\overline{i},\overline{j}}(X) =(p-n+1)(n-1)e_{n-1}(X)
\end{eqnarray*}
since the last sum consists of products of length $n-1$ and every product appears $(p-n+1)(n-1)$ times. Indeed, if we fix a product $X_{i_1}X_{i_2}\cdot\ldots\cdot X_{i_{n-1}}$ of length $n-1$,  it appears in $X_i e_{n-2}^{\overline{i},\overline{j}}(X)$ if and only if $i\in \{i_1,i_2,\ldots,i_{n-1}\}$ and $j\notin \{i_1,i_2,\ldots,i_{n-1}\}$. Consequently, we can choose $i$ on $n-1$ ways and $j$ on $p-(n-1)$ ways. It implies that \pg{the drift part of $e_n(X)$ equals}  $ (p-n+1)(\alpha-n+1)e_{n-1}(X)dt$. In order to show \eqref{eq:poly:SDE} and  \eqref{eq:poly:bracks}, it remains to show that 
\begin{eqnarray}
\label{simple_brack}
  \sum_{\pg{i=1}}^p X_{\pg{i}} e_{n-1}^{\overline{\pg{i}}}(X)e_{m-1}^{\overline{\pg{i}}}(X) = \sum_{k=1}^{{\pg p}} (m-n+2k-1)e_{n-k}(X)e_{m+k-1}(X)
\end{eqnarray}
for every $1\leq n\leq m\leq p$ \pg{(recall the notation $e_r\equiv0$ if $r<0$ or $r>p$)}. Observe that both sides of \eqref{simple_brack} are symmetric polynomials of degree $m+n-1$, where the variables \pg{$X_1,\ldots, X_p$} appear at most in power $2$. Due to symmetry, it is enough to show that, for a fixed ${\pg{l}\ge 0}$ and $j\geq 1$, the expression
\begin{eqnarray*}
   X_1^2\cdot\ldots\cdot X_{\pg{l}}^2X_{{\pg{l}}+1}\cdot\ldots\cdot X_{{\pg{l}}+j}
\end{eqnarray*}
appears on both sides of \eqref{simple_brack} the same number of times. Here $2{\pg{l}}+j=n+m-1$. Moreover, by the form of the LHS of \eqref{simple_brack}, we have  $l\le n-1$ and, consequently, $l+j =n-1-l+m\ge m\ge n$. The quadratic expression $X_1^2\cdot\ldots\cdot X_{\pg{l}}^2$ can only appear on the left-hand side of \eqref{simple_brack}  from the multiplication of $e_{n-1}^{\overline{\pg{i}}}(X)$ and $e_{m-1}^{\overline{\pg{i}}}(X)$ and $X_1\cdot\ldots\cdot X_{\pg{l}}$ must appear in both of them. Thus, it remains to count
  \pg{in  how many terms  of the LHS  the factors $X_{l+1},\ldots, X_{l+j}$   appear, so that the product
  $X_1^2\cdot\ldots\cdot X_{\pg{l}}^2X_{{\pg{l}}+1}\cdot\ldots\cdot X_{{\pg{l}}+j}$ is obtained. }
  
  \pg{Let  $s_i=X_{\pg{i}} e_{n-1}^{\overline{\pg{i}}}(X)e_{m-1}^{\overline{\pg{i}}}(X)$ be a term   of the left-hand side of \eqref{simple_brack}.
  	Observe that obligatorily  $X_i\in\{X_{l+1},\ldots, X_{l+j} \}$.
  Thus there are $j$ possible  choices of a term $s_i$.
We fix such a choice and count  the terms of the polynomial $e_{n-1}^{\overline{\pg{i}}}(X)$, which contain the product
$X_1\cdot\ldots\cdot X_{\pg{l}}$ and have remaining $n-1-l$ variables in the set $\{X_{l+1},\ldots, X_{l+j} \}\setminus \{X_i\}$. Equivalently, we count all choices of $n-1-l$ elements in a set with $j-1$ elements.  The remaining factors of $X_1^2\cdot\ldots\cdot X_{\pg{l}}^2X_{{\pg{l}}+1}\cdot\ldots\cdot X_{{\pg{l}}+j}$ come from the polynomial $e_{m-1}^{\overline{\pg{i}}}(X)$.
}
 Finally the coefficient of $X_1^2\cdot\ldots\cdot X_{\pg{l}}^2X_{{\pg{l}}+1}\cdot\ldots\cdot X_{{\pg{l}}+j}$ on the LHS of  \eqref{simple_brack} is
\begin{eqnarray*} 
   j \binom{j-1}{n-1-l} = (n-l)\binom{j}{n-l}
\end{eqnarray*}
\pg{(recall that $1\le n-l\le j$).}
Similarly, the considered product $X_1^2\cdot\ldots\cdot X_{\pg{l}}^2X_{{\pg{l}}+1}\cdot\ldots\cdot X_{{\pg{l}}+j}$
 appears in $e_{n-k}(X)e_{m+k-1}(X)$ exactly $\binom{j}{n-k-l}$ times.  Thus, it is enough to show that \pg{ for $j,l,m,n$ satisfying  $1\le n-l\le j$ and } $2l+j=n+m-1$,
  the following \pg{combinatorial} identity holds: 
\begin{eqnarray*}
   (n-l)\binom{j}{n-l} = \sum_{k=1}^{n} (m-n+2k-1)\binom{j}{n-k-l}\/.
\end{eqnarray*}
We use a convention that \pg{the Newton's symbol} $\binom{n}{\pg{r}}$ is zero whenever $r>n$ or $r<0$.

Using the relation $2l+j=m+n-1$, we  can rewrite the right-hand side as 
\begin{eqnarray*}
  \sum_{k=1}^{n} (m-n+2k-1)\binom{j}{n-k-l} &=& \sum_{k=1}^{n-l} (j-2(n-l-k))\binom{j}{n-l-k}.
\end{eqnarray*}
Substitutions $N=n-l-1$  \pg{and $  r=n-l-k$}  together with reordering the sum lead to 
\pg{a combinatorial formula}
\begin{eqnarray}\label{known}
   \sum_{r=0}^{N}(j-2r)\binom{j}{\pg{r}}= (N+1)\binom{j}{N+1},
\end{eqnarray}
\pg{where $0\le N\le j-1$}.
Formula \eqref{known} is known (see e.g. \cite{bib:DiscreteMath}) 
andvcan be easily proved by elementary induction on $N$.
\end{proof}

\section{ \pg{Proofs of Theorems 1-3}   }
\label{sec:EandU}

\begin{proof}[Proof of Theorem \ref{thm:SDEs:NC:unicity}]
Since we consider all possible starting points $x_1\leq \ldots\leq x_p$ (without restriction that $x_1$ must be non-negative), we can and we do assume that $\alpha\geq 0$. The general case follows immediately by multiplying equations \eqref{eq:BESQp:SDE} by $-1$ and re-ordering the particles.

First we note that the conditions $(C1)$ and $(A1)$ (or equivalently $(A1')$) from \cite{bib:gm2} hold for functions $\sigma(x)=2\sqrt{|x|}$, $b(x)=\alpha$ and $H(x,y)=|x|+|y|$.
\pg{For $(A1)$, see the proof of \cite[Cor. 6.5]{bib:gm2}.}

\pg{By  Theorem 5.3 and Remark 2.4} in \cite{bib:gm2} 
we get the pathwise uniqueness for non-colliding solutions 
(the other assumptions in Theorem 5.3 of \cite{bib:gm2}  were  used to construct such non-colliding solution). Consequently, it is enough to prove the existence of a non-colliding solution. If $\alpha\notin\{0,1,\ldots,p-2\}$, the result follows directly from Theorem 2.2 from \cite{bib:gm2} (see also Corollary 6.6 therein, \pg{where $\R$ should be $\R^+$}). 

Thus we focus on $\alpha\in\{0,1,\ldots,p-2\}$ and consider the general starting point $x=(x_1,\ldots,x_p)$.
\pg{Recall that condition (A4) from \cite{bib:gm2} fails if $\alpha\in\{0,1,\ldots,p-2\}$.} For simplicity, we denote $\textrm{rk}^+(x)=n$, $\textrm{rk}^-(x)=l$ and $m=p-\textrm{rk}(x)$, i.e. 
\begin{eqnarray*}
x_1\leq\ldots\leq  x_{l}<0=x_{l+1}=x_{l+2}=\ldots=x_{l+m}<x_{l+m+1}\leq \ldots\leq x_p\/.
\end{eqnarray*}
Recall  that $n^*$ is defined \pg{in \eqref{def:n_star}  }as  an integer  such that $2n^* \in\{p+\alpha,p+\alpha+1\}$. Note that  $\alpha\leq n^*< p$ since $\alpha\leq p-2$.

 Now we consider two cases. 
 
\textbf{Case 1: $n\leq n^*$ and $l\leq p-n^*$.} In this case we construct a solution by glueing two independent processes. First, we set $p_{-}=p-n^*>0$ and $\alpha_{-}=n^*-\alpha\ge 0$ and consider \pg{a system of $p_-$ SDEs}
\begin{eqnarray*}
   dZ_i= 2\sqrt{|Z_i|}dB_i +\left(\alpha_{-}+\sum_{j=1,j\neq i}^{p_{-}}\frac{|Z_i|+|Z_j|}{Z_i-Z_j}\right)dt,\quad i=1,\ldots, p_-
\end{eqnarray*}
starting from $Z_i(0)=-x_{p-\pg{n^*}-i+1}$ for $i=1,\ldots,p_{-}$.

 Note that our assumption $n\leq n^*$ implies $p-n^*\leq p-n = i+j$ and consequently $Z=(Z_1,\ldots,Z_{p_{-}})$ starts from non-negative point, i.e. $Z_1(0)=-x_{p-n^*}\geq 0$. Moreover we have $\alpha_{-}\geq p_{-}$, since $2n^*\geq p+\alpha$. It guaranties, by our previous considerations, that there exists unique strong solution which is non-colliding and the solution is non-negative ($\alpha_{-}\geq p_{-}$). Then, we put $p_+=n^*$ and $\alpha_+=\alpha+p-n^*$ and consider \pg{a system of $p_+$ SDEs}
\begin{eqnarray*}
   dY_i= 2\sqrt{|Y_i|}dB_i +\left(\alpha_+
   \pg{+}
   \sum_{j={p-n^*+1},j\neq i}^{p}\frac{|Y_i|+|Y_j|}{Y_i-Y_j}\right)dt\/,\quad i=p-n^*+1,\ldots,p\\
\end{eqnarray*}
where $Y_i(0)=x_i$ for $i=p-n^*+1, \ldots,p$. Once again our assumption $l\leq p-n^*$ ensures that $p-n^*+1\geq l+1$ and consequently the considered starting point is non-negative, i.e. $x_{p-n^*+1}\geq 0$. Moreover, we have $\alpha_+\geq p_+-1$ since $2n^*\leq p+\alpha+1$, which means that there exists unique strong non-colliding solution which is also non-negative. Now we put
\begin{eqnarray*}
   X_i(t) &=& \left\{
\begin{array}{rcl}
-Z_{p-n^*-i+1}(t) && i=1,\ldots,p-n^*\\
Y_i(t) && i=p-n^*+1,\ldots,p
\end{array}
\right.
\end{eqnarray*} 
and obviously we have $X_i(0)=x_i$ for every $i=1,\ldots,p$. Moreover, for every $i=1,\ldots,p-n^*$ and $j=p-n^*+1,\ldots,p$ we have
\begin{eqnarray*}
   \frac{|X_i|+|X_j|}{X_i-X_j} = -1\/,\quad \frac{|X_j|+|X_i|}{X_j-X_i}=1
\end{eqnarray*}
since $X_i(t)\leq 0$ and $X_j(t)\geq 0$. It implies that for $i=1,\ldots,p-n^*$ we can write
\begin{eqnarray*}
   dX_i &=& 2\sqrt{|X_i|}dB_i+\left(\alpha-n^*+\sum_{j=1,j\neq i}^{p-n^*}\frac{|X_i|+|X_j|}{X_i-X_j}\right)dt\\
	&=& 2\sqrt{|X_i|}dB_i + \left(\alpha + \sum_{j=1,j\neq i}^p\frac{|X_i|+|X_j|}{X_i-X_j}\right)dt 
\end{eqnarray*}
and the analogous computations can be done for remaining $i=p-n^*+1,\ldots,p$. Note also that $X=(X_1,\ldots,X_p)$ is non-colliding. Indeed, as we have seen, there are no collisions between $X_1,\ldots,X_{p-n^*}$ and separately between $X_{p-n^*+1},\ldots,X_p$. Moreover, the first particle system is non-positive and the other is non-negative, i.e. $X_{p-n^*}(t)\leq 0\leq X_{p-n^*+1}(t)$ for every $t>0$ a.s. It remains to show that these two particles do not collide at zero. However, if $2n^*=p+\alpha+1$, then $\alpha_{-}=n^*-\alpha=p-n^*+1=p_{-}+1$ and consequently $X_{p-n^*}(t)<0$ for every $t>0$. If $2n^*=p+\alpha$ then we have $\alpha_{-}=p_{-}$ and $\alpha_+=p_+$ which implies  that particles $X_{p-n^*}$ and $X_{p-n^*+1}$ visit zero but the sets $\{t:X_{p-n^*}(t)=0\}$ and $\{t:X_{p-n^*+1}(t)=0\}$ are of Lebesgue measure zero (see Proposition 4 in \cite{bib:b91}). In particular, there exists sequence $t_i\searrow 0$ such that $X_{p-n^*}(t_i)>0$ a.s. and consequently, there are no collisions at every $t_i$. By Proposition 4.2 in \cite{bib:gm2} we know that the particles will never collide after $t_n$ and thus there are no collisions for every $t>0$.

\textbf{ Case 2: $n>n^*$ or $l>p-n^*$.}
Following \pg{the main idea of} \cite{bib:gm2}, we get a solution, \pg{solving first} the SDEs for the elementary symmetric polynomials, i.e. we use a solution $e=(e_1,\ldots,e_p)$ of \eqref{eq:poly:general}. We set $(X_1,\ldots,X_p)=f(e_1,\ldots,e_p)$, where $f$ is the \pg{diffeomorphism} described in Section \ref{sec:poly}. It remains to show that 
$(X_1,\ldots,X_p)$ is non-colliding. If $m\leq 1$, i.e. there is at most one particle starting from zero, the result follows directly from the first part of the proof of Proposition 4.3 in \cite{bib:gm2}. Thus, it is enough to show that if $m>1$, the particles starting from zero will exit that point just after the start. By continuity of the paths, we have $\tau_1>0$ a.s. where $\tau_1=\inf\{t>0: X_i(t)=0\}\wedge \inf\{t>0: X_{i+j+1}(t)=0\}$, i.e. we do not have any additional zero particle up to time $\tau_1$. Assume that all particles starting from zero remain at zero for some $\tau_2>0$ with positive probability and put $\tau=\tau_1\wedge \tau_2$. Then it is clear that $e_{n}(X)\equiv 0$ for $t<\tau$, where $n=i+k+1$, since every product of length $n$ \pg{contains} at least one zero particle. In particular, the drift of $e_n(X)$ vanishes for $t<\tau$, but from the other side, it is equal to
\begin{eqnarray*}
    \textrm{drift}[e_n] &=& \sum_{i=1}^p\alpha e_{n-1}^{\overline{i}}(X)-\sum_{i< j}(|X_i|+|X_j|)e_{n-2}^{\overline{i}, \overline{j}}(X)\\
		&=& m e_{n-1}(X) \left(\alpha+l-n\right)dt\/.
\end{eqnarray*}
Indeed, for $t<\tau$, we have $e_{n-1}^{\overline{i}}(X)\equiv 0$ if $X_i(t)=0$ and $e_{n-1}^{\overline{i}}(X)=e_{n-1}(X)$ (the product of all non-zero particles) if $X_i(t)\neq 0$. Moreover, the expression $(|X_i|+|X_j|)e_{n-2}^{\overline{i}, \overline{j}}(X)$ is non-zero only if exactly one of particles $X_i$, $X_j$ is zero and
\begin{eqnarray*}
   \sum_{i< j}(|X_i|+|X_j|)e_{n-2}^{\overline{i}, \overline{j}}(X) &=& \sum_{i=l+1}^{l+m}\sum_{j=1}^p |X_j|e_{n-2}^{\overline{i}, \overline{j}}(X)\ind_{\{X_j\neq 0\}} = \sum_{i=l+1}^{l+m}\sum_{j=1}^p |X_j|\frac{e_{n-1}(X)}{X_j}\ind_{\{X_j\neq 0\}}\\
	&=& me_{n-1}(X)\sum_{j=1}^p \textrm{sgn}(X_j) =  me_{n-1}(X)(l-n)\/.
\end{eqnarray*}
However, if $n>n^*$ then $2n> p+\alpha$ and consequently $\alpha+l-n=\alpha+l+n-2n< l+n-p\leq 0$. On the other hand, if $l>p-n^*$, then $\alpha+l-n> p-n+\alpha-n^*\geq 0$, since $n\leq p$ and $\alpha\leq n^*$. In both cases we have $\alpha+l-n \neq 0$. It leads to a contradiction since $e_{n-1}(X)$ does not vanish as a product of non-zero particles. It means that at least one zero particle must become non-zero immediately. It will increase the number of non-zero particles on $\{t<\tau_1\}$ and consequently we will still have $n'>n^*$ or $l'>p-n^*$, where $l'$ and $n'$ are numbers of strictly negative and positive particles after instant exit from zero of some particles. Thus we can proceed using strong Markov property and inductively show that all particles must leave zero just after the start. This ends the proof.
\end{proof}
In fact, the above-given proof leads directly to the result presented in Theorem \ref{thm:SDEs:general}.
\begin{proof}[Proof of Theorem \ref{thm:SDEs:general}]
Existence of a solution was proved in Theorem \ref{thm:SDEs:NC:unicity}. Thus, it is enough to show that any solution of \eqref{eq:BESQp:SDE} is non-colliding. Then, using uniqueness of non-colliding solutions proved in Theorem \ref{thm:SDEs:NC:unicity}, we get the result. Thus let $X=(X_1,\ldots,X_p)$ be a solution. Then by It\^o formula and the computations provided in Proposition 3.1 in \cite{bib:gm2} we claim that the SDEs for $e_n(X)$ are of the same form but with $|X_i|+|X_j|$ replaced by $(|X_i|+|X_j|)\ind_{\{X_i\neq X_j\}}$. However, it does not affect the arguments presented above in the proof of Theorem 1, which say that whenever $\alpha\notin \{0,\ldots,p-2\}$ or $\alpha\in\{0,\ldots,p-2\}$ but $\rk^+(x)>n^*$ or $\rk^{-}(x)<p-n^*$ the particles become immediately distinct and never collide again. Note that adding the indicators $\ind_{\{X_i\neq X_j\}}$ does not affect conditions $(A1)$, $(A3)$, $(A4)$ and $(A5)$ needed in \cite{bib:gm2} and used above. The condition $(A2)$, which here simplifies to 
\begin{eqnarray*}
  |x|+|y|\leq (|x|+|y|)\ind_{\{x\neq y\}}\/,
\end{eqnarray*}
holds for every $x\neq y$, but it is enough for Theorem 4.4 from \cite{bib:gm2} \pg{to be true}. 

To finish the proof we construct a solution for $\alpha\in\{0,\ldots,p-2\}$ starting from $x=(x_1,\ldots,x_p)$ such that $\rk^+(x)\leq n^*$  and $\rk^{-}(x)\leq p-n^*$, which is not non-colliding, i.e. the uniqueness of a solution does not hold. First we note that there exist integers $n< n^*$ and $l< p-n^*$ such that $\alpha+l-n=0$. Let $Z=(Z_1,\ldots,Z_n)$ be the process $BESQ^{(\alpha^+)}_{nc}(x_1,\ldots,x_n)$, where $\alpha^+ = \alpha+p-n$,  described by 
\begin{eqnarray*}
  dZ_i  = 2\sqrt{|Z_i|}dB_i+\left(\alpha^+ +\sum_{j\neq i}\frac{|Z_i|+|Z_j|}{Z_i-Z_j}\right)dt\/,\quad i=1,\ldots,n\/.
\end{eqnarray*}
Note that $\alpha^+> n-1$ ($n<n^*<p+\alpha+1$) and consequently $Z$ is non-negative. Moreover, set $\alpha^{-}=-(\alpha-p+l)$ and let $Y=(Y_1,\ldots,Y_l)$ be $BESQ_{nc}^{\alpha^{-}}(-x_p,-x_{p-1},\ldots,-x_{p-l})$, i.e. 
\begin{eqnarray*}
  dY_i=2\sqrt{|Y_i|}dB_{p-i+1}+\left(\alpha^{-}+\sum_{j\neq i}\frac{|Y_i|+|Y_j|}{Y_i-Y_j}\right)dt\/,\quad i=1,\ldots,l\/.
\end{eqnarray*}
As previously, we have $\alpha^{-}>l-1$ and $Y$ is non-negative. Now we glue these solutions together with $p-n-l$ particles constantly equal to zero, i.e. we set 
\begin{equation*}
  X_i = \left\{
	\begin{array}{cc}
	   Z_i\/,& i=1,\ldots,n\\
		 0\/,& i=n+1,\ldots,p-l\\
		 -Y_{p+1-i}\/,& i=p-l+1.\ldots,p 
	\end{array}
	\right.
\end{equation*}
We can easily check that $X=(X_1,\ldots,X_p)$ solves 
\begin{eqnarray*}
  dX_i = 2\sqrt{|X_i|}dB_i+\left(\alpha+\sum_{j\neq i}\frac{|X_i|+|X_j|}{X_i-X_j}\ind_{X_i\neq X_j}\right)dt\/.
\end{eqnarray*}
Indeed, since $X_1,\ldots,X_{p-l}$ are non-negative and $X_{n+1},\ldots,X_p$ are non-positive we have
\begin{eqnarray*}
  \frac{|X_i|+|X_j|}{X_i-X_j} &=& 1 \textrm{ for } i=1,\ldots,n\/,\quad j=n+1,\ldots,p\/,\\
	 \frac{|X_i|+|X_j|}{X_i-X_j} &=& -1 \textrm{ for } i=p-l+1,\ldots,p\/,\quad j=1,\ldots,p-l\/,
\end{eqnarray*}
and the drift parts for $i=1,\ldots,n$ and $i=p-l+1,\ldots,p$ are reduced to those for $Z$ and $Y$ respectively. Moreover, for $i=n+1,\ldots,p-l$ the drift part is just $\alpha-n+l$ which is zero as we have assumed.

Finally, we show that $X=(X_1,\ldots,X_p)$ has collisions after the start. Note that since $n<n^*$ and $l<p-n^*$ then $n+l<p$. If $n+l<p-1$, then there are at least two zero particles, i.e. they collide for every $t>0$. If $n+l=p-1$, i.e. we have exactly one particle constantly equal to zero, then 
$\alpha^+<n+1$ or $\alpha^{-}<l+1$. Indeed, if $\alpha^+=\alpha+p-n\geq n+1$ and $\alpha^{-}=p-\alpha-l\geq l+1$, then summing these inequalities we get $2p-2(n+l)\geq 2$. Thus $X_n$ or $X_{p-l+1}$ hits zero with probability $1$, i.e. we have a collision between one of these particles and $X_{n+1}\equiv 0$. This ends the proof. 
\end{proof}

\begin{proof}[Proof of Theorem 3]
If $\alpha\geq p-1$ then, by Theorem \ref{thm:SDEs:general}, there exists unique strong solution, which is non-colliding (by Theorem \ref{thm:SDEs:NC:unicity}). Moreover, by  Theorem \ref{thm:poly}, the product $e_p$ of the particles is the time-changed one-dimensional squared Bessel process of non-negative index $\alpha-p+1$ starting from non-negative point. Consequently, it remains non-negative and since the particles are separate after the start it implies that the solution is non-negative. Moreover, for $\alpha\in \{0,1,\ldots,p-2\}$ and $\rk(x)\leq \alpha$, the non-negative solution was also constructed in \cite{bib:b89}, see also \cite{bib:gmm:2017}. Note that one can construct such solution in the same way as in the proof of the previous theorem by letting $l=0$.

Assume that there exists a non-negative solution $(X_1,\ldots,X_p)$ for $\alpha<p-1$ but not in $\{0,1,\ldots,p-2\}$  or $\alpha\in \{0,1,\ldots,p-2\}$ but $\textrm{rk}(x_0)>\alpha$. Then there are at least $\alpha+1$ particles different from $X_1$ on some positive time interval $[0,T]$, $T>0$. In the first case we have only non-colliding solution, so all the particles are different, in the other case we just use the continuity of the paths. In both cases the drift of $X_1$ can be estimated as follows
\begin{eqnarray*}
  \textrm{drift}(X_1) = \alpha+\sum_{j=2}^p\frac{|X_1|+|X_j|}{X_1-X_j}\ind_{\{X_i\neq X_j\}}\leq \alpha-(\alpha+1)\leq -1\/.
\end{eqnarray*}
Here we used the simple inequality $(|x|+|y|)/(x-y)\leq -1$ valid for every $x<y$. Consequently, by the comparison theorem and the fact that $BESQ^{(-1)}{(X_1(0))}$ becomes strictly negative on every time interval with positive probability we get a contradiction with our initial assumption that $X_1$ is non-negative.

Thus, it remains to show that for $\alpha\in \{0,1,\ldots,p-2\}$ and $\rk(x)\leq \alpha$ the solution is unique among non-negative solutions.
We show that the first $p-\alpha$ particles of non-negative solutions must stay at zero. Indeed, if at any time there are more than $\alpha$ particles different from $X_1$, then we go back to the above-described situation (the rank of the starting point is too large) and using Strong Markov Property we can conclude that the solution becomes negative with positive probability. Consequently $X_1(t)=\ldots=X_{p-\alpha}(t)$ for every $t\geq 0$. Moreover, if $X_1$ becomes non-zero at some time, then by results of \cite{bib:gm2}, the solution immediately becomes non-colliding and there are $p-1$ particles different from $X_1$. Once again, by Strong Markov Property, we get that $X_1$ becomes negative with positive probability. Finally, knowing that $X_1(t)=\ldots=X_{p-\alpha}(t)=0$ for every $t$, the equations for the remaining $X_{p-\alpha+1}, \ldots,X_{p}$ are 
\begin{eqnarray*}
   dX_i = 2\sqrt{|X_i|}dB_i + \left(p+\sum_{\stackrel{j=p-\alpha+1,\ldots,p}{j\neq i}}\frac{|X_i|+|X_j|}{X_i-X_j}\ind_{\{X_i\neq X_j\}}\right)dt\/,\quad i=p-\alpha+1,\ldots,p\/.
\end{eqnarray*}
Note that  this is just the system of SDEs describing $\tilde{p}=\alpha$ particles with index $\tilde{\alpha}=p$. Since $\tilde{\alpha}>\tilde{p}+1$ there exists unique non-negative solution, which ends the proof. 
\end{proof}

\section{The structure of \pg{non-colliding systems} $BESQ_{nc}^{\pg{(\alpha)}}(x_1,\dots,x_p)$}
\label{sec:structure}
G\"oing-Jaeschke and Yor in \cite{bib:gyj} studied the structure of squared Bessel processes with negative indices. They showed that $BESQ^{(-\alpha)}(x)$ starting from positive $x$ with $\alpha>0$ hits zero almost surely and then behaves as $-BESQ^{(\alpha)}(0)$. In this section we will study the corresponding problem for \pg{non-colliding} squared Bessel particles systems $BESQ_{nc}^{(\alpha)}(x_1,\ldots,x_p)$. The negativity of the index in the classical case is translated to the condition $\alpha<p-1$ and we assume that $0\leq x_1\leq \ldots\leq x_p$. 
We define the family of first hitting times
\begin{eqnarray*}
   T^{(i)}_0 = \inf\{t\geq 0: X_i(t)=0\}\/,\quad i=1,\ldots,p\/.
\end{eqnarray*}
and the family of first entrance times
 \begin{eqnarray*}
   T^{(i)}_{-} = \inf\{t\geq 0: X_i(t)<0\}\/,\quad i=1,\ldots,p\/.
\end{eqnarray*}
In the next proposition we generalize the well-known fact saying that $BESQ^{(\alpha)}(x)$ hits zero whenever $\alpha\in[0,2)$, visits negative half-line for $\alpha<0$ and stays non-positive after first entrance to the negative half-line. We also describe the evolution of the solution between the moments when the succeeding particles become negative.

\begin{theorem}
\label{thm:structure}
  Let $X=(X_1,\ldots,X_p)$ be $BESQ_{nc}^{(\alpha)}(x_1,\ldots,x_p)$, where $0\leq x_1\leq x_2\leq \ldots \leq x_p$ and $\alpha< p+1$. Let 
$n=\left \lceil{\frac{p+1-\alpha}{2}}\right \rceil $. Then 
\begin{eqnarray*}
T_0^{(1)}\leq T_0^{(2)}\leq \ldots\leq T^{(n)}_0<\infty\/,\quad T_0^{(n+1)}=\ldots=T_0^{(p)}=\infty
\end{eqnarray*}
and
\begin{eqnarray*}
T_{-}^{(1)}\leq\ldots\leq  T^{(n-1)}_{-}<\infty\/,\quad T_{-}^{(n)}=\ldots=T_{-}^{(p)}=\infty\/.
\end{eqnarray*}
Moreover, for every $k=1,\ldots,n-1$, on the interval $[T_{-}^{(k)},T_{-}^{(k+1)})$ the \pg{subsystems of particles} $Y_k=(X_1,\ldots,X_{k})$ and $Z_k=(X_{k+1},\ldots,X_p)$ are conditionally independent given $\left(Y_k(T_{-}^{(k)}), Z_k(T_{-}^{(k)})\right)$ and they evolve as $-BESQ_{nc}^{(p-\alpha-k)}$ and $BESQ_{nc}^{(\alpha+k)}$ respectively. 

In particular, if $T_{-}^{(i)}$ is finite then $X_i(t)\leq 0$ for $t\geq T_{-}^{(i)}$, i.e. the particles do not go back to the positive half-line after going below zero.  
\end{theorem}
\begin{remark}
   Note that for given $p$ and $\alpha<p+1$ the number $n=\left \lceil{\frac{p+1-\alpha}{2}}\right \rceil $ is $1$ for $\alpha\in [p-1,p+1)$, $n=2$ for $\alpha\in [p-3,p-1)$ and so on. Consequently, the above-given result states that the $i$th particle hits zero if and only if $p+3-\alpha>2i$ and the $i$th particle visits negative half-line if and only if $p+1-\alpha>2i$.
\end{remark}
\begin{remark}
Since the system becomes non-colliding immediately, we can have $T_0^{(i)}=T_0^{(i+1)}$ or $T_{-}^{(i)}=T_{-}^{(i+1)}$ only if $x_{i}=x_{i+1}=0$. Consequently, if $x_i>0$ or $x_i<x_{i+1}$ then we have strict inequalities between times $T_{0}^{(i)}$ and $T_{0}^{(i+1)}$ (analogously $T_{-}^{(i)}<T_{-}^{(i+1)}$) in the above-given theorem. 
\end{remark}
\begin{proof}[Proof of Theorem \ref{thm:structure}]
Let $(X_1,\ldots,X_p)$ be a non-colliding solution to \eqref{eq:BESQp:SDE} with given Brownian motions $(B_1,\ldots,B_{\pg{p}})$. Bru in \cite{bib:b91} showed that for $\alpha\in\pg{(}p-1,p+1)$, the first particle hits zero almost surely ($T_0^{(1)}<\infty$), but it remains non-negative ($T_{-}^{(1)}=\infty$).

For $\alpha\leq p-1$ we define $\tilde{X}_1$ as a solution to the following SDE
\begin{eqnarray*}
   d\tilde{X}_1 = 2\sqrt{|\tilde{X}_1|}dB_1+(\alpha-p+1)dt
\end{eqnarray*}
starting from $x_1$. This process is $BESQ^{(\alpha-p+1)}(x_1)$ driven by the same Brownian motion as $X_1$. Following the proof of the comparison theorem (see Theorem 3.7, p.394 in \cite{bib:ry99}), \pg{we} notice that the local time at zero $L^0(\tilde{X}_1-X_1)$ vanishes and consequently, using the Ta\pg{n}aka's formula, we can write
\begin{eqnarray*}
   \ex{(X_1-\tilde{X}_1)^+} = \ex \int_0^t \ind_{\{X_1(s)>\tilde{X}_1(s)\}}\left(p-1+\sum_{i=2}^p\frac{|X_1(s)|+|X_i(s)|}{X_1(s)-X_i(s)}\right)ds\leq 0\/.
\end{eqnarray*}
The last inequality follows from a simple observation that $(|x|+|y|)/(x-y)\leq -1$ for $y>x$. Thus $X_1(t)\leq \tilde{X}_1(t)$ for every $t\geq 0$ a.s. This implies that $X_1$ hits zero. Moreover, for $\alpha<p-1$ the process becomes negative ($T_0^{(1)}=T_{-}^{(1)}<\infty$) and remains non-positive for $t>T_0^{(1)}$, because the same holds for the squared Bessel process $\tilde{X}_1$ with negative index $\alpha-p+1$. For $\alpha=p-1$ the process is non-negative (by Theorem \ref{thm:SDEs:general} and \ref{thm:SDEs:Nonnegative}, i.e. the unique non-colliding solution is non-negative), i.e. $T_{-}^{(1)}=\infty$.

To examine the behaviour of the system after the time $T_{-}^{(1)}$ (for $\alpha<p-1$), we define $X_i^{*}(t)=X_i(T_{-}^{(1)}+t)$ and $B^*_i(s) = B_i(T^{(1)}_{-}+s)-B_i(T^{(1)}_{-})$ for $i=1,\ldots,p$. Note that, by strong Markov property, the process $(B_1^*,\ldots,B_p^*)$ is again a $p$-dimensional Brownian motion and in particular $B^*_i$ are independent. Moreover, we have $X^*(0)=0$ and for $t<T_{-}^{(1)}-T_{-}^{(2)}$ we have
\begin{eqnarray*}
  X^*_1(t) &=&  \int_{T^{(1)}_{-}}^{T^{(1)}_{-}+t}2\sqrt{|{X_1(s)}|}dB_1(s)+t\alpha+ \int_{T^{(1)}_{-}}^{T^{(1)}_{-}+t}\sum_{k=2}^p\frac{|X_1(s)|+|X_k(s)|}{X_1(s)-X_k(s)}ds\\
  & = &\int_{T^{(1)}_{-}}^{T^{(1)}_{-}+t}2\sqrt{|{X_1(s)}|}dB_1(s)+(\alpha-p+1)t = \int_0^t 2\sqrt{|X^*_1(s)|}d B^*_1(s)+(\alpha-p+1)t\/,
\end{eqnarray*}
where we used the fact that $(|x|+|y|)/(x-y)=-1$ whenever $x\leq 0\leq y$. Similarly, for $i=2,\ldots,p$ we get
\begin{eqnarray*}
  X^*_i(t) -X^{*}_i(0)& = &\int_{T^{(1)}_{-}}^{T^{(1)}_{-}+t}2\sqrt{|{X_i(s)}|}dB_i(s)+t\alpha
  + \int_{T^{(1)}_{-}}^{T^{(1)}_{-}+t}\sum_{k\neq i}^p\frac{|X_i(s)|+|X_k(s)|}{X_i(s)-X_k(s)}ds\\
  & =& \int_{T^{(1)}_{-}}^{T^{(1)}_{-}+t}2\sqrt{|{X_i(s)}|}dB_i(s)+t(\alpha+1)+\int_{T^{(1)}_{-}}^{T^{(1)}_{-}+t}\sum_{k>1, k\neq i }\frac{|X_i(s)|+|X_k(s)|}{X_i(s)-X_k(s)}ds\\
  & =& \int_0^t 2\sqrt{|X^*_i(s)|}dB^*_i(s)+t(\alpha+1)+\int_{0}^{t}\sum_{k>1, k\neq i }\frac{X^{*}_i(s)+X^{*}_k(s)}{X^*_i(s)-X^*_k(s)}ds\/.
\end{eqnarray*}
Note that the interactions between particles $X_1^*$ and $X_2^*,\ldots,X_p^*$ disappeared from the corresponding drift parts and, consequently, the processes $Y_1=X_1$ and $Z_1=(X_2,\ldots,X_p)$ on $[T_{-}^{(1)},T_{-}^{(2)})$ are conditionally independent, given the starting point $Z_1(T_{-}^{(1)})$. Moreover, $Y_1$ is $-BESQ^{(p-1-\alpha)}(0)$ and $Z_1$ evolves as a non-colliding squared Bessel system of $p-1$ particles with index $\alpha+1$. 

By strong Markov property, we can apply the above-given argument to the system of $p^*=p-1$ particles $(X_2^*,\ldots,X_p^*)$ with index $\alpha^*=\alpha+1$ and show that if $\alpha<p-3$ (\pg{which is equivalent to }$\alpha^*<p^*-1$) then $T_{-}^{(2)}<\infty$. Moreover, after going into $(-\infty,0]$ the second particle becomes invisible (independent) for the non-negative particles, but starts to interact with the first one. Indeed, we have
 \begin{eqnarray*}
  \bar{X}_i(t) -\bar{X}_i(0)& = &\int_0^t 2\sqrt{|\bar{X}_i(s)|}d\bar{B}_i(s)+t(\alpha+2)+\int_{0}^{t}\sum_{k>2, k\neq i }\frac{\bar{X}_i(s)+\bar{X}_k(s)}{\bar{X}_i(s)-\bar{X}_k(s)}ds\/.
\end{eqnarray*}
for $i=3,4,\ldots,p$ and 
\begin{eqnarray*}
  \bar{X}_j(t) &=&  \int_0^t 2\sqrt{|\bar{X}_j(s)|}d \bar{B}_j(s)+(\alpha-p+2)t\/,\quad j=1,2\/,
\end{eqnarray*}
where $\bar{X}(t)=X(T_{-}^{(2)}+t)$ and $\bar{B}(t)=B(T_{-}^{(2)}+t)-B(T_{-}^{(2)})$. 

\pg{We complete the proof by iterating} this procedure.
 When $\alpha$ is small enough the consecutive particles become negative and then the non-negative and non-positive \pg{particle subsystems} evolve independently as  squared Bessel particle systems
 \pg{with appropriate drift parameters}.

\end{proof}




\begin{thebibliography}{00}


\bibitem{bib:b89}
M.~F.~Bru,
\emph{Diffusions of perturbed principal component analysis}. 
J. Multivariate Anal. 29 (1989), no. 1, 127-136. 

\bibitem{bib:b91}
M.~F.~Bru,
\emph{Wishart processes.}
J. Theor. Prob. 4 (1991) 725--751.

\bibitem{bib:donatiyor} 
C. Donati-Martin,Y.  Doumerc, H. Matsumoto, M. Yor, 
\emph{Some properties of the Wishart processes and a matrix extension of the Hartman-Watson laws}.  
Publ. Res. Inst. Math. Sci. 40 (2004), no. 4, 1385-1412. 

\bibitem{bib:gyj}
A.~G\"oing-Jaeschke, M.~Yor,
\emph{A survey and some generalizations of Bessel processes.}
Bernoulli 9 (2003), no. 2, 313-349.

\bibitem{bib:gm1}
P. Graczyk, J. Ma\l{}ecki, 
\emph{Multidimensional Yamada-Watanabe theorem and its applications to particle systems.}
J, Math. Phys. 54  (2013), 021503, 15pp.

\bibitem{bib:gm2}
P. Graczyk, J. Ma\l{}ecki, 
\emph{Strong solutions of non-colliding particle systems.}
Electron. J. Probab. 19 (2014), no. 119, pp. 1-21.

\bibitem{bib:gmm:2017}
P. Graczyk, J. Ma\l{}ecki, E. Mayerhofer,
\emph{Characterizations of Wishart processes and Wishart distributions.}
to appear in Stoch. Processes Appl. (2017).

\bibitem{bib:katori2016}
M. Katori,
\emph{Bessel processes, Schramm-Loewner evolution, and the Dyson model.}
Springer, Tokyo, 2016.

\bibitem{bib:katori2004}
M. Katori and H. Tanemura,
\emph{Symmetry of matrix-valued stochastic processes and non-colliding diffusion particle systems.}
Journal of Mathematical Physics, 45 (8) (2004) 3058--3085.

\bibitem{koc} W. K\"onig, N. O'Connell
\emph{Eigenvalues  of  the  Laguerre  process  as  noncolliding  squared  Bessel  process.}
Elect. Comm. in Probab. 6 (2001) 107-114.

\bibitem{bib:ry99}
D.~Revuz, M.~Yor
\emph{Continuous martingales and Brownian motion.}
Springer, New York, 1999.

\bibitem{bib:DiscreteMath}  
K.~A.~Ross and C.~R.~B.~Wright
\emph{Discrete Mathematics} (5th ed.) Prentice Hall, 2003.

\bibitem{bib:yw71}
T.~Yamada, S.~Watanabe
\emph{On the uniqueness of solutions of stochastic differential equations.}
J. Math. Kyoto Univ. 11 (1971) 155--167.

\end{thebibliography}
\end{document}